\documentclass[reqno,12pt]{amsart}
\usepackage{amsfonts}

\usepackage{amsmath,amssymb,amsthm,amsfonts}
\usepackage{bbm}
\usepackage[colorlinks]{hyperref}

\usepackage{color}

\newtheorem{lemma}{Lemma}[section]
\newtheorem{theorem}{Theorem}[section]

\newtheorem{proposition}{Proposition}[section]

\newtheorem{corollary}{Corollary}[section]

\numberwithin{equation}{section}

\newcommand{\R}{\mathbb{R}}

\newcommand{\T}{\mathbb{T}}

\newcommand{\CE}{\mathcal{E}}

\def\v{\varepsilon}

\def\n{\nu}

\def\b{\beta}

\def\l{\lambda}

\def\r{\rho}

\def\f{\frac}

\begin{document}

\title[Global existence for the ES-BGK model]{Global Existence for the Ellipsoidal BGK Model with Initial Large Oscillations}

\author[R.-J. Duan]{Renjun Duan}
\address[R.-J. Duan]{Department of Mathematics, The Chinese University of Hong Kong, Shatin, Hong Kong}
\email{rjduan@math.cuhk.edu.hk}

\author[Y. Wang]{Yong Wang}
\address[Y. Wang]{Institute of Applied Mathematics, AMSS, CAS, Beijing 100190, P.R.~China, and Department of Mathematics, The Chinese University of Hong Kong, Shatin, Hong Kong}
\email{yongwang@amss.ac.cn}

\author[T. Yang]{Tong Yang}
\address[T. Yang]{Department of Mathematics, City University of Hong Kong, Hong Kong}
\email{matyang@cityu.edu.hk}

\date{\today}

\maketitle
\begin{abstract}
The ellipsoidal BGK model was introduced in \cite{Ho} to fit the correct Prandtl number in the Navier-Stokes approximation of the classical BGK model. In the paper we establish the global existence of mild solutions to the Cauchy problem on the model for a class of initial data allowed to have large oscillations. The proof is motivated by a recent study of the same topic on the Boltzmann equation in \cite{DHWY}.    
\end{abstract}


\thispagestyle{empty}

\section{Introduction}

In this paper, we consider the Ellipsoidal BGK (ES-BGK for short) model
\begin{equation}\label{1.1}
F_t+v\cdot\nabla_x F=A_{\nu}\Big(M_{\n}-F\Big).
\end{equation}
Here the unknown $F=F(t,x,v)\geq 0$ denotes the density distribution function of gas particles which have position $x\in\Omega=\mathbb{R}^3$ or $\mathbb{T}^3$ 
and  velocity $v\in\mathbb{R}^3$  at  time $t>0$.  Corresponding to a given parameter $\nu\in (-1/2,1)$ related to the Prandtl number of the above Boltzmann-type  model (cf.~\cite{ALPP}), $A_\n=A_\n(F)$ is the collision frequency and $M_\n=M_\n(F)$ is the anisotropic Gaussian, both depending on the unknown function $F$; their explicit forms will be given later on. We refer readers  to \cite{BGK,Wa} and \cite{Ho} for the origin and background of the ES-BGK model, \cite{ABLP, FJ,GT}  for the numerical investigations of the model, and \cite{Yun,Yun1} for the recent mathematical studies on  the existence of solutions.   

For given $F(t,x,v)$, we introduce the usual fluid quantities density $\r$, velocity $u$, temperature $T$ and stress tensor $\Theta$, respectively, as 
\begin{align*}
\r(t,x)&=\int_{\mathbb{R}^3}F(t,x,v)dv,\\
u(t,x)&=\frac{1}{\r}\int_{\mathbb{R}^3}vF(t,x,v)dv,\\
T(t,x)&=\frac{1}{3\r}\int_{\mathbb{R}^3}|v-u|^2F(t,x,v)dv,\\
\Theta(t,x)&=\frac{1}{\r}\int_{\mathbb{R}^3}(v-u)\otimes(v-u) F(t,x,v)dv,
\end{align*}
and further denote the 
tensor $\mathcal{T}_\n$ by 
\begin{equation}\label{1.3}
\mathcal{T}_\n=(1-\n) T\,\mathrm{Id}+\n\Theta,
\end{equation}
where $\mathrm{Id}$ is the $3\times3$ identity matrix. Then the collision frequency $A_\n=A_\n(F)$ and the anisotropic Gaussian $M_\n=M_\n(F)$ are respectively defined by  
\begin{align*}
M_{\n}
=\f{\r}{\sqrt{\mbox{det}(2\pi\mathcal{T}_\n)}}\exp\Big(-\f12(v-u)^{t}\mathcal{T}_\n^{-1}(v-u)\Big),
\end{align*}
and  
\begin{eqnarray*}
A_\n=\f{\r T}{1-\n}.
\end{eqnarray*}

We are interested in the well-posedness of the Cauchy problem on the 
equation \eqref{1.1} with initial 
data
\begin{equation}\label{1.6}
F(t,x,v)|_{t=0}=F_0(x,v).
\end{equation}
The right-hand collision term in the ES-BGK model 
satisfies (cf.\cite{ABLP,ALPP})
\begin{align*}
\int_{\mathbb{R}^3}\left(M_\n-F\right)\left(\begin{array}{c}1\\ v \\ |v|^2 \end{array} \right)dv=\left(\begin{array}{c}0\\0\\0\end{array} \right).
\end{align*}
This implies that for any solution $F(t,x,v)$, one has conservations of defect mass, defect momentum and defect energy as in \cite{Guo}: 
\begin{align*}
\int_{\Omega}\int_{\mathbb{R}^3}(F(t,x,v)-\mu(v))\,dvdx&= \int_{\Omega}\int_{\mathbb{R}^3}(F_0(x,v)-\mu(v))\,dvdx\triangleq M_0,\\
\int_{\Omega}\int_{\mathbb{R}^3}v(F(t,x,v)-\mu(v))\,dvdx&= \int_{\Omega}\int_{\mathbb{R}^3}v(F_0(x,v)-\mu(v))\,dvdx\triangleq J_0,\\
\int_{\Omega}\int_{\mathbb{R}^3}|v|^2(F(t,x,v)-\mu(v))\,dvdx&= \int_{\Omega}\int_{\mathbb{R}^3}|v|^2(F_0(x,v)-\mu(v))\,dvdx\triangleq E_0,
\end{align*}
for all $t>0$, where 
\begin{equation*}
\mu(v)=\f{1}{(2\pi)^{\f32}}\exp\left(-\f{|v|^2}{2}\right)
\end{equation*}
is the normalized global Maxwellian. As shown in \cite{ALPP}, 
the entropy dissipation property also holds:
\begin{equation*}
\f{d}{dt}\int_{\Omega}\int_{\mathbb{R}^3}(F\ln{F}-\mu\ln{\mu})dvdx\leq 0,
\end{equation*}
for all $t>0$. 
For later use, as in \cite{DHWY}, due to Proposition \ref{prop2.4}  we denote
\begin{equation*}
\mathcal{E}(F_0):= \int_{\Omega}\int_{\mathbb{R}^3}F_0\ln F_0-\mu\ln\mu \,dvdx+[\f{3}{2}\ln(2\pi)-1]M_0+\f12E_0.	
\end{equation*}
Note that one has $\mathcal{E}(F_0)\geq 0$ for any $F_0(x,v)\geq 0$, cf.~\cite{DHWY}.

\

The main result of the paper is stated as follows.

\begin{theorem}\label{thm1.1}
Let $w(v)=(1+|v|)^{\b}$ with  $\b>7$. Assume that $F_0(x,v)\geq 0$, $\|w F_0\|_{L^\infty}<+\infty$, and
\begin{equation}\label{1.17}
\inf\limits_{t\geq 0,x\in \Omega}\int_{\mathbb{R}^3} F_0(x-vt,v)dv\geq C_0,
\end{equation}
for a positive constant $C_0>0$. There are  constants $\v_0>0$, $\tilde{C}_1\geq 1$ and $\tilde{C}_2\geq1$ such that if 
\begin{equation}
\label{ad.i1}
\mathcal{E}(F_0)+\sup_{t\geq t_1, x\in \Omega}e^{-\f{t}{1-\nu}}\Big|\int_{\mathbb{R}^3}(1,v,|v|^2, v\otimes v)[F_0(x-vt,v)-\mu(v)]dv \Big|\leq \v_0,
\end{equation}
with $t_1:=(\tilde{C}_1\|w F_0\|_{L^\infty})^{-1}$,
then the Cauchy problem \eqref{1.1}, \eqref{1.6}  of the ES-BGK model admits a unique global-in-time mild solution $F(t,x,v)\geq 0$ such that
\begin{equation*}
\sup_{t\geq 0} \|w F(t)\|_{L^\infty}\leq \tilde{C}_2,
\end{equation*}
where $\v_0$,  $\tilde{C}_1$ and  $\tilde{C}_2$ depend only on $\n, ~C_0$ and $\|w F_0\|_{L\infty}$.
\end{theorem}

It should be pointed out that the initial data under the assumptions of Theorem \ref{thm1.1} are allowed to have large oscillations in the spatial variable. For example, one may take 
\begin{align*}
F_0(x,v)=\r_0(x)\mu(v),
\end{align*}
with $0<C^{-1}\leq \r_0(x)\leq C$ and $\|\r_0-1\|_{L^1_x}\ll 1$. Then it is straightforward to check that $F_0(x,v)$ satisfies all the conditions of Theorem \ref{thm1.1}. 
Indeed, first of all, it is easy to observe that  $\|w F_0\|_{L^\infty}$ is finite, \eqref{1.17} is true and also it holds that
\begin{equation*}
\mathcal{E}(F_0)\leq  \|\r_0\ln\r_0-\r_0+1\|_{L^1_x}+C\|\r_0-1\|_{L^1_x}.
\end{equation*}
To estimate the second term on the left of \eqref{ad.i1}, we divide into two cases as in \cite{DHWY}. 

\medskip
\noindent{\it Case $\Omega=\R^3$}: For each $t\geq t_1$ and $x\in \R^3$, one has 
\begin{multline*}
\int_{\mathbb{R}^3}(1+|v|^2)|F_0(x-vt,v)-\mu(v)|\,dv
\leq C\int_{\mathbb{R}^3}|\r_0(x-vt)-1|\,dv\\
= Ct^{-3}\int_{\mathbb{R}^3}|\r_0(y)-1|\,dy\leq Ct_1^{-3}\|\r_0-1\|_{L^1_x}.
\end{multline*}

\medskip
\noindent{\it Case $\Omega=\T^3$}: For each $t\geq t_1$ and $x\in \T^3$,  it holds that
\begin{multline*}
\int_{\mathbb{R}^3}(1+|v|^2)|F_0(x-vt,v)-\mu(v)|\,dv\leq \int_{\R^3}|\rho_0(x-vt)-1| (1+|v|^2)\mu(v)\,dv\\
\leq C\int_{|v|\geq N} (1+|v|^2)\mu(v)\,dv+ \frac{C[1+(Nt)^3]}{t^3} \int_{\T^3} |\rho_0(y)-1|\,dy,\end{multline*}
and hence
\begin{equation*}
\int_{\mathbb{R}^3}(1+|v|^2)|F_0(x-vt,v)-\mu(v)|\,dv\leq C(N) +C (t_1^{-3}+N^3)\|\r_0-1\|_{L^1_x},
\end{equation*}
where $C(N)$ tends to zero as $N$ goes to infinity. 

\medskip
\noindent Recall that $t_1>0$ depends only on the upper bound of $\r_0$. Therefore,  with the above estimates, one can see that the smallness assumption on \eqref{ad.i1} in Theorem \ref{thm1.1} can be satisfied by requiring $ \|\r_0\ln\r_0-\r_0+1\|_{L^1_x}+\|\r_0-1\|_{L^1_x}$ to be small. Under this situation, it is easy to see that initial data can be allowed to have large oscillations. In fact, \eqref{ad.i1} should contain much more general initial data with large oscillations.

\medskip

The rest of the paper is arranged as follows. We give some basic lemmas in Section 2, and then establish the local  $L^\infty$ estimates and global $L^\infty$ estimates in Section 3 and Section 4, respectively. Thus Theorem \ref{thm1.1} immediately follows by the same argument as in \cite{DHWY}.


\medskip

\noindent{\it Notations.}  Throughout this paper, $C$ denotes a generic positive constant which may vary from line to line.  
$\|\cdot\|_{L^2}$ denotes the standard $L^2(\Omega\times\mathbb{R}^3_v)$-norm, and $\|\cdot\|_{L^\infty}$ denotes the $L^\infty(\Omega\times\mathbb{R}^3_v)$-norm.


\section{Preliminaries}

We need some useful inequalities (cf.~\cite{Perthame}) in the following lemma stating the lower bounds of velocity-weighted $L^\infty$ norms of $F(t,x,v)$. 

\begin{lemma}\label{lem2.1}
Let $N_q(F)=\|(1+|v|)^{q} F(t,x,v)\|_{L^\infty_{x,v}}$, then it holds that 
\begin{align}
&\f{\r}{T^{\f32}}\leq CN_0(F),\label{2.1}
\\
&\r(T+|u|^2)^{\f{q-3}{2}}\leq C_q N_q(F),~q>5~~\mbox{or}~~0\leq q<3,\notag
\\
&\f{\r|u|^q}{((T+|u|^2)T)^{\f32}}\leq C_q N_q(F),~q>1,\notag
\end{align}
where $C$ and $C_q$ are constants  independent of $F$.
\end{lemma}

The below lemma whose proof can be found in \cite{Yun} gives the relation between the temperature tensor $\mathcal{T}_{\n}$ and the scalar temperature function $T$.

\begin{lemma}\label{lem2.2}
Let $-1/2<\nu<1$, and define
\begin{equation*}
C_{\n1}\triangleq \min\{1-\nu,1+2\n\},~~\mbox{and}~~C_{\n2}\triangleq \max\{1-\nu,1+2\n\}.
\end{equation*}
Then, if the density function $\r(t,x)>0$, it holds that 
\begin{equation*}
C_{\n1}T(t,x) \,\mathrm{Id}\leq \mathcal{T}_{\n}\leq C_{\nu2} T(t,x) \,\mathrm{Id}.
\end{equation*}
\end{lemma}

It follows from Lemma \ref{lem2.2} that one has

\begin{corollary}\label{cor2.3}
 Assume $0<T(t,x)<\infty$, then it holds that
 \begin{equation*}
 C_{\n2}^{-1} T^{-1}(t,x)\, \mathrm{Id}\leq \mathcal{T}_\n^{-1}\leq  C_{\n1}^{-1} T^{-1}(t,x) \,\mathrm{Id},
 \end{equation*}
 and 
 \begin{equation*}
 C_{\n1}^3T^3(t,x)\leq \mathrm{det}\,\mathcal{T}_{\n}\leq  C_{\n2}^3T^3(t,x).
 \end{equation*}
\end{corollary}

For the later proof, we also need the following proposition whose proof can be found in \cite{Guo} or \cite{DHWY}.


\begin{proposition}\label{prop2.4}
Let $F(t,x,v)$ be the solution to \eqref{1.1} and \eqref{1.6}, then it holds that 
\begin{align*}
	&\int_{\Omega}\int_{\mathbb{R}^3}\f{|F(t,x,v)-\mu(v)|^2}{4\mu(v)}I_{\{|F(t,x,v)-\mu(v)|\leq \mu(v)\}}dvdx\nonumber\\
	&\quad+\int_{\Omega}\int_{\mathbb{R}^3}\f{1}{4}|F(t,x,v)-\mu(v)|I_{\{|F(t,x,v)-\mu(v)|\geq \mu(v)\}}dvdx\nonumber\\
	&\leq \int_{\Omega}\int_{\mathbb{R}^3}F_0\ln F_0-\mu\ln\mu \,dvdx+[\f{3}{2}\ln(2\pi)-1]M_0+\f12E_0=
	\mathcal{E}(F_0).
	\end{align*}

\end{proposition}


\section{Local Estimates}

The mild form of the ES-BGK model equation \eqref{1.1}  can be written as 
\begin{eqnarray}
F(t,x,v)&=&F_0(x-vt,v) e^{-\int_0^tA_{\n}(\tau,x-v(t-\tau))d\tau}\notag\\
&&+\int_0^t e^{-\int_s^tA_{\n}(\tau,x-v(t-\tau))d\tau}A_{\n}(s,x-v(t-s))\notag\\
&&\qquad\qquad\qquad\qquad\qquad\times M_{\n}
(s,x-v(t-s),v)ds. \label{3.1}
\end{eqnarray}
Based on the mild formulation, one can obtain the a priori  estimates on the velocity-weighted $L^\infty$ norms of solutions in a short strictly positive time for initial data of possibly large oscillations.

\begin{lemma}\label{lem3.1}
Let $w(v)=(1+|v|)^\beta$ with $\beta>5$, and $t_1\triangleq (4C_\b(\n)\|w F_0\|_{L^\infty})^{-1}>0$, then it holds that
\begin{align*}
\|w F(t)\|_{L^\infty}\leq 2\|w F_0\|_{L^\infty},
\end{align*}
for all $t\in[0,t_1]$, where $C_\b(\n)$ is 
an explicitly computable constant depending only on $\b$ and $\n$.
\end{lemma}

\begin{proof}
It follows from \eqref{3.1} that
\begin{multline}\label{3.3}
|w(v)F(t,x,v)|
\leq  
\|w(v)F_0(v)\|_{L^\infty}\\
+\int_0^t A_{\n}(s,x-v(t-s))\cdot w(v) M_{\n}
(s,x-v(t-s),v)ds.
\end{multline}
In order to estimate the last integral term of \eqref{3.3}, one first notices that 
\begin{align}\label{3.4}
A_\n(s,x-v(t-s))&=\f{1}{1-\n}(\r T)(s,x-v(t-s))\nonumber\\
&\leq \f{1}{3(1-\n)}(3\r T+\f{1}{2}|u|^2)(s,x-v(t-s))\nonumber\\
&\leq \f{1}{3(1-\n)}\int_{\mathbb{R}^3}|\eta|^2F(s,x-v(t-s),\eta)d\eta\nonumber\\
&\leq C_\beta(\n)\sup_{y,\eta}\, (1+|\eta|)^\beta F(s,y,\eta),
\end{align}
due to $\beta>5$, where $C_\beta(\nu)$ which may vary from line to line is a generic constant depending only on $\beta$ and $\nu$. Moreover, it follows from Lemma \ref{lem2.1} and Corollary \ref{cor2.3} that
\begin{align}\label{3.5}
M_\n(s,y,v)
&\leq C(\n)\f{\r}{T^{\f32}}\exp\left(-\f{|v-u|^2}{C(\nu)T}\right) \leq C(\n)\f{\r}{T^{\f32}}\leq C(\nu)\|F(s)\|_{L^\infty},
\end{align}
and 
\begin{align}\label{3.6}
|v|^\b M_\n(s,y,v)
&\leq C_\beta |u|^\b M_\n(s,y,v)
+C_\beta|v-u|^\b M_\n(s,y,v)
\nonumber\\
&\leq C_\beta(\nu)\f{\r}{T^{\f32}}|u|^\b+C_\beta(\nu)\r T^{\f{\b-3}2}\nonumber\\
&\leq C_\beta (\n)\|(1+|v|)^{\b}F(s)\|_{L^\infty}.
\end{align}
Combining \eqref{3.5} and \eqref{3.6}, one gets that 
\begin{equation}\label{3.7}
|w(v) M_\n(s,y,v)
|\leq C_\b(\nu) \|w F(s)\|_{L^\infty}.
\end{equation}
Then it follows from \eqref{3.4} and \eqref{3.7} that for $y=x-v(t-s)$,
\begin{align}\label{3.8}
\Big|A_{\n}(s,y
)\cdot w(v) M_{\n}
(s,y
,v) \Big|\leq C_\b(\n)\|w F(s)\|^2_{L^\infty}.
\end{align}
Substituting \eqref{3.8} into \eqref{3.3}, one obtains that for all $t\geq 0$, 
\begin{align*}
\|w F(t)\|_{L^\infty}\leq \|w F_0\|_{L^\infty}+C_\b(\n)\int_0^t\|w F(s)\|^2_{L^\infty}ds.
\end{align*}
Choosing $t_1=(4C_\b(\n)\|w F_0\|_{L^\infty})^{-1}>0$, it is straightforward to verify by the continuity argument that
\begin{align*}
\sup_{0\leq t\leq t_1}\|w F(t)\|_{L^\infty}\leq 2\|w F_0\|
_{L^\infty}.
\end{align*}
Thus the proof of Lemma \ref{lem3.1} is complete.
\end{proof}


As a consequence of Lemma \ref{lem3.1}, 
one can further obtain some bounds on the macroscopic variables which will be used in the later proof.

\begin{lemma}\label{lem3.2}
Let $w(v)=(1+|v|)^\beta$ with $\beta>5$.  Assume that there is $C_0>0$ such that 
$F_0(x,v)\geq 0$ satisfies 
\begin{equation*}
\int_{\mathbb{R}^3} F_0(x-vt,v)dv\geq C_0,
\end{equation*}
for all $t\geq 0$ and $x\in \Omega$.
Then it holds that  for all $0\leq t\leq t_1$ and $x\in \Omega$,
\begin{align}
C_1^{-1}\leq \r(t,x),~T(t,x)\leq C_1,~~\mbox{and}~~|u(t,x)|\leq C_1, \label{3.13}
\end{align}
where $t_1>0$ is given in Lemma \ref{lem3.1}, and $C_1\geq 1$ is an explicitly computable constant  
depending only on $C_0$, $\nu$, $\beta$ and $\|w F_0\|_{L^\infty}$.
\end{lemma}

\begin{proof}
First notice by Lemma \ref{lem3.1} that for $0\leq t\leq t_1$ and $x\in \Omega$,
\begin{align}
\r(t,x)&=\int_{\mathbb{R}^3}F(t,x,v)dv\leq C\|w F(t)\|_{L^\infty}\leq C\|w F_0\|_{L^\infty},\label{3.16}\\
|(\r u)(t,x)|&=\Big|\int_{\mathbb{R}^3}vF(t,x,v)dv\Big|\leq C\|w F(t)\|_{L^\infty}\leq C\|w F_0\|_{L^\infty},\label{3.18}\\
 (
 \r|u|^2+3\r T)(t,x)&=\int_{\mathbb{R}^3}
 |v|^2F(t,x,v)dv\leq C\|w F(t)\|_{L^\infty}\leq C\|w F_0\|_{L^\infty},
\label{3.17}
\end{align}
where $\beta>5$ has been used. For the lower bound of density $\r(t,x)$, it follows from \eqref{3.1} and \eqref{3.17} that 
\begin{align}\label{3.19}
\r(t,x)&=\int_{\mathbb{R}^3}F(t,x,v)dv\geq\int_{\mathbb{R}^3}e^{-\int_0^tA_{\n}(\tau,x-v(t-\tau))d\tau}F_0(x-vt,v)dv\nonumber\\
&\geq \exp{\left(-\int_0^t\f{\|(\r T)(s)\|_{L^\infty}}{1-\n}ds\right)}\int_{\mathbb{R}^3}F_0(x-vt,v)dv\nonumber\\
&\geq C_0 \exp{\left(-\int_0^t\f{\|w F_0\|_{L^\infty}}{C(1-\n)}ds\right)}\geq C_0 \exp{\left(-\f{t_1\|w F_0\|_{L^\infty}}{C(1-\n)}\right)}\nonumber\\
&=C_0 \exp{\left(-\f{1}{4C C_\b(\n)(1-\n)}\right)}.
\end{align}
Furthermore, it follows from \eqref{2.1}   and \eqref{3.16}, \eqref{3.18}, \eqref{3.17}, \eqref{3.19} that 
\begin{align*}
T(t,x)+|u(t,x)|\leq C\|w F_0\|_{L^\infty}C_0^{-1} \exp{\left(\f{1}{4C C_\b(\n)(1-\n)}\right)},
\end{align*}
and 
\begin{align*}
T(t,x)\geq \left(\f{\r(t,x)}{C\|w F_0\|_{L^\infty}}\right)^{\f23}\geq \left(\f{C_0}{C\|w F_0\|_{L^\infty}} \exp{\left(-\f{1}{4C C_\b(\n)(1-\n)}\right)}\right)^{\f23}.
\end{align*}
Therefore those estimates in \eqref{3.13} follow by defining 
\begin{align*}
C_1&\triangleq \max\Big\{C\|w F_0\|_{L^\infty}, ~~ \Big(\f{C_0}{C\|w F_0\|_{L^\infty}} \exp{\Big(-\f{1}{4C C_\b(\n)(1-\n)}\Big)}\Big)^{-\f23},\nonumber\\
 &~~~~~~~~~~~~~C_0^{-1}\exp{\Big(\f{1}{4C C_\b(\n)(1-\n)}}\Big), ~~C\|w F_0\|_{L^\infty}C_0^{-1} \exp{\Big(\f{1}{4C C_\b(\n)(1-\n)}\Big)}\Big\}.
\end{align*}
The proof of Lemma \ref{lem3.2} is complete.
\end{proof}

\section{Global Estimates}

In this section, 
we consider the global-in-time estimates on the solution $F(t,x,v)$ to the Cauchy problem \eqref{1.1}, \eqref{1.6} under the following a priori assumptions:
\begin{equation}\label{4.1}
(2C_1)^{-1}\leq \r(t,x),~T(t,x)\leq 2C_1,~~\mbox{and}~~|u(t,x)|\leq 2C_1,
\end{equation}
for all $t\geq t_1$ and $x\in \Omega$,  where $t_1>0$ and $C_1>0$ are respectively given in Lemma \ref{lem3.1} and Lemma \ref{lem3.2}. First of all, we have

\begin{lemma}\label{lem4.1}
Let $w(v)=(1+|v|)^\beta$ with $\beta>5$. Under the assumption \eqref{4.1}, it holds that 
\begin{align}\label{4.2}
\|w F(t)\|_{L^\infty}\leq \|w F_0\|_{L^\infty}+C(\n)C_1^{\f{13}2},
\end{align}
for all $t\geq 0$, where $C(\nu)>0$ is a constant depending only on $\n$.
\end{lemma}

\begin{proof}   
It follows from \eqref{2.1}, \eqref{3.5} and \eqref{4.1} that 
\begin{align*}
|w F(t,x,v)|&\leq \|w F_0\|_{L^\infty}+\int_0^t e^{-\f{t-s}{4C^2_1(1-\nu)}} \f{4C_1^2}{1-\nu}M_\nu
(s,x-v(t-s),v)ds\nonumber\\
&\leq \|w F_0\|_{L^\infty}+C(\n)\int_0^t\f{4C_1^2}{1-\nu} e^{-\f{t-s}{4C^2_1(1-\nu)}} \f{\r}{T^{\f32}}ds\nonumber\\
&\leq \|w F_0\|_{L^\infty}+C(\n)C_1^{\f92}\int_0^t\f{1}{1-\nu} e^{-\f{t-s}{4C_1^2(1-\nu)}} ds\nonumber\\
&\leq \|w F_0\|_{L^\infty}+C(\n)C_1^{\f{13}2},
\end{align*}
which gives \eqref{4.2}. 
The proof of Lemma \ref{lem4.1} is complete. 
\end{proof}

To close the a priori assumption \eqref{4.1}, we need the following 
lemma whose proof is based on the reformulated mild form by \eqref{3.1}:
\begin{eqnarray}
F(t,x,v)-\mu&=&[F_0(x-vt,v)-\mu] e^{-\int_0^tA_{\n}(\tau,x-v(t-\tau))d\tau}\notag\\
&&+\int_0^t e^{-\int_s^tA_{\n}(\tau,x-v(t-\tau))d\tau}A_{\n}(s,x-v(t-s))\notag\\
&&\qquad\qquad\qquad\qquad\times [M_{\n}
(s,x-v(t-s),v)-\mu]ds. \label{4.4}
\end{eqnarray}

\begin{lemma}\label{lem4.2}
Let $w(v)=(1+|v|)^\beta$ with $\beta>7$.
It holds that
\begin{align}\label{4.5}
&\sup_{x\in\Omega}\Big|\int_{\mathbb{R}^3}(1,v,|v|^2, v\otimes v)[F(t,x,v)-\mu(v)]dv\Big|\nonumber\\
&\leq e^{-\f{t}{1-\nu}}\sup_{x\in\Omega}\Big|\int_{\mathbb{R}^3}(1,v,|v|^2, v\otimes v)[F_0(x-vt,v)-\mu(v)]dv \Big|\nonumber\\
&~~~+C\Big\{\sqrt{\mathcal{E}(F_0)}+\mathcal{E}(F_0)^{\f{\b-5}{2\b-7}}\Big\},
\end{align}
for all $t\geq 0$,  
where $C\geq1$ is a constant depending only on $C_0, \nu$, $\beta$ and $\|w F_0\|_{L^\infty}$.
\end{lemma}

\begin{proof}
In fact, \eqref{4.4} gives
\begin{align}\label{4.6}
&\Big|\int_{\mathbb{R}^3}[F(t,x,v)-\mu(v)]dv\Big|\nonumber\\
&\leq \Big|\int_{\mathbb{R}^3}[F_0(x-vt,v)-\mu(v)]\cdot \exp{\Big(-\int_0^tA_{\n}(\tau,y)d\tau\Big)}dv\Big|\nonumber\\
&\quad+\int_0^te^{-\f{t-s}{4C^2_1(1-\nu)}}\f{4C_1^2}{1-\nu}\int_{\mathbb{R}^3}(1+|v|^2)|M_\n(F)(s,y,v)-\mu(v)|dvds,
\end{align}
with $y:=x-v(t-s)$.  Notice that 
\begin{align}
&[F_0(x-vt,v)-\mu(v)]\cdot \exp{\Big(-\int_0^tA_{\n}(\tau,y)d\tau\Big)}\nonumber\\
&=[F_0(x-vt,v)-\mu(v)]\cdot \exp{\Big(-\int_0^t\f{1}{1-\nu}d\tau\Big)}\nonumber\\
&\quad+[F_0(x-vt,v)-\mu(v)]\cdot\left\{\exp{\Big(-\int_0^tA_{\n}(\tau,y)d\tau\Big)}- \exp{\Big(-\int_0^t\f{1}{1-\nu}d\tau\Big)}\right\}.\nonumber
\end{align}
This implies that 
\begin{align}\label{4.6-2}
&\Big|\int_{\mathbb{R}^3}[F_0(x-vt,v)-\mu(v)]\cdot \exp{\Big(-\int_0^tA_{\n}(\tau,y)d\tau\Big)}dv\Big|\nonumber\\
&\leq e^{-\f{t}{1-\nu}}\Big|\int_{\mathbb{R}^3}[F_0(x-vt,v)-\mu(v)]dv \Big|\nonumber\\
&~~~~+\f{C}{1-\nu}e^{-\f{t}{4C^2_1(1-\nu)}}(1+\|w F_0\|_{L^\infty})\int_0^t \int_{\mathbb{R}^3}w(v)^{-1}|(\r T)(\tau,y)-1|dvd\tau\nonumber\\
&\leq e^{-\f{t}{1-\nu}}\Big|\int_{\mathbb{R}^3}[F_0(x-vt,v)-\mu(v)]dv \Big|\nonumber\\
&~~~~+\f{C}{1-\nu}e^{-\f{t}{4C^2_1(1-\nu)}}\int_0^t \int_{\mathbb{R}^3}w(v)^{-1}\left|\int_{\mathbb{R}^3}(1,\eta, |\eta|^2)[F(s,y,\eta)-\mu(\eta)]d\eta\right|dvds.
\end{align}
It is direct to further compute
\begin{align}
&|M_\nu
(s,y,v)-\mu|\notag\\
&\leq C\Big(|\r(s,y)-1|+|u(s,y)|+|\mathcal{T}_\nu(s,y)-\mathrm{Id}|\Big)e^{-\f{|v|^2}{C}}\nonumber\\
&\leq C\Big(|\r(s,y)-1|+|(\r u)(s,y)|+|\r\mathcal{T}_\nu(s,y)-\mathrm{Id}|\Big)e^{-\f{|v|^2}{C}}.\label{4.7}
\end{align}
Recall \eqref{1.3}. One then can write
\begin{align}\label{4.8}
\r\mathcal{T}_\nu&=(1-\nu)\r T Id+\nu \r \Theta\nonumber\\
&=\f{1-\nu}{3}\int_{\mathbb{R}^3}|\eta-u|^2[F(s,y,\eta)-\mu]d\eta\notag\\
&\qquad+\nu\int_{\mathbb{R}^3}(\eta-u)\otimes(\eta-u) [F(s,y,\eta)-\mu]d\eta\nonumber\\
&\qquad+\f{1-\nu}{3}\int_{\mathbb{R}^3}|\eta-u|^2\mu d\eta+\nu\int_{\mathbb{R}^3}(\eta-u)\otimes(\eta-u) \mu d\eta
\nonumber\\
&=\f{1-\nu}{3}\int_{\mathbb{R}^3}|\eta-u|^2[F(s,y,\eta)-\mu]d\eta\notag\\
&\qquad+\nu\int_{\mathbb{R}^3}(\eta-u)\otimes(\eta-u) [F(s,y,\eta)-\mu]d\eta\nonumber\\
&\qquad+Id+\f{1-\nu}{3} |u|^2Id+\nu u\otimes u,
\end{align}
where we have denoted $u=u(s,y)$ on the right. Using \eqref{4.8}, it follows from \eqref{4.7} that
\begin{align}
|M_\nu(F)(s,y,v)-\mu|
&\leq C\Big(|\r(s,y)-1|+|(\r u)(s,y)|\Big)\cdot e^{-\f{|v|^2}{C}}\nonumber\\
&\quad+Ce^{-\f{|v|^2}{C}}\Big|\int_{\mathbb{R}^3}(1,\eta, |\eta|^2,\eta\otimes\eta)\{F(s,y,\eta)-\mu\}d\eta\Big| \nonumber\\
&\leq Ce^{-\f{|v|^2}{C}}\Big|\int_{\mathbb{R}^3}(1,\eta, |\eta|^2,\eta\otimes\eta)[F(s,y,\eta)-\mu]d\eta\Big|,\nonumber
\end{align}
which together with \eqref{4.6} and \eqref{4.6-2},  yield that
\begin{align}
&\Big|\int_{\mathbb{R}^3}[F(t,x,v)-\mu(v)]dv\Big|\nonumber\\
&\leq C\int_0^te^{-\f{t-s}{4C^2_1(1-\nu)}}\f{4C_1^2}{1-\nu} \int_{\mathbb{R}^3}w(v)^{-1}\left|\int_{\mathbb{R}^3}(1,\eta, |\eta|^2,\eta\otimes\eta)[F(s,y,\eta)-\mu(\eta)]d\eta\right|dvd\tau \nonumber\\
&\qquad+e^{-\f{t}{1-\nu}}\Big|\int_{\mathbb{R}^3}[F_0(x-vt,v)-\mu(v)]dv \Big|.\nonumber
\end{align}
By similar arguments as above, one can obtain that 
\begin{align}\label{4.10}
&\Big|\int_{\mathbb{R}^3}(1,v,|v|^2, v\otimes v)[F(t,x,v)-\mu(v)]dv\Big|\nonumber\\
&\leq \int_0^tCe^{-\f{t-s}{4C^2_1(1-\nu)}}\f{4C_1^2}{1-\nu} \int_{\mathbb{R}^3}w(v)^{-1}|v|^2\left|\int_{\mathbb{R}^3}(1,\eta, |\eta|^2,\eta\otimes\eta)[F(s,y,\eta)-\mu]d\eta\right|dvd\tau \nonumber\\
&\qquad+e^{-\f{t}{1-\nu}}\Big|\int_{\mathbb{R}^3}(1,v,|v|^2, v\otimes v)[F_0(x-vt,v)-\mu(v)]dv \Big|\nonumber\\
&=: e^{-\f{t}{1-\nu}}\Big|\int_{\mathbb{R}^3}(1,v,|v|^2, v\otimes v)[F_0(x-vt,v)-\mu(v)]dv \Big|\nonumber\\
&\qquad+\int_{t-\l}^t(\cdots)d\tau+\int_0^{t-\l}(\cdots)d\tau,
\end{align}
where $\l>0$ is a small constant to  
be chosen later. It remains to  
estimate the  second and third terms on the right of \eqref{4.10}. For the second term, we have 
\begin{align}\label{4.11}
\int_{t-\l}^t(\cdots)ds\leq C\l\sup_{y\in\mathbb{R}^3,t-\l\leq s\leq t }\left|\int_{\mathbb{R}^3}(1,\eta, |\eta|^2,\eta\otimes\eta)[F(s,y,\eta)-\mu]d\eta\right|.
\end{align}
Now we estimate the third term on the right of \eqref{4.10}.
For the case $\Omega=\mathbb{R}^3$, we divide the integral into two parts:
\begin{align}\label{4.12-3}
&\int_{\mathbb{R}^3}w(v)^{-1}|v|^2\left|\int_{\mathbb{R}^3}(1,\eta, |\eta|^2,\eta\otimes\eta)[F(s,y,\eta)-\mu]d\eta\right|dv\nonumber\\
&\leq\int_{\mathbb{R}^3}\int_{\mathbb{R}^3}(1+|v|)^{-\b+2}(1+|\eta|^2)|[F(s,y,\eta)-\mu]|\cdot I_{\{|F(s,y,v)-\mu|\leq\mu\}}d\eta dv\nonumber\\
&\qquad{\small+\int_{\mathbb{R}^3}\int_{\mathbb{R}^3}(1+|v|)^{-\b+2}(1+|\eta|^2)|[F(s,y,\eta)-\mu]| \cdot I_{\{|F(s,y,v)-\mu|\geq\mu\}}d\eta dv}\nonumber\\
&=:J_1+J_2.
\end{align}
For terms on the right of \eqref{4.12-3}, we have, for $\b>7$,  that 
\begin{align}\label{4.12-4}
J_1&\leq C\Big(\int_{\mathbb{R}^3}\int_{\mathbb{R}^3}(1+|v|)^{-2\b+4}e^{-\f12|\eta|^2}(1+|\eta|^2)^2d\eta dv\Big)^{\f12}\nonumber\\
&\qquad\times\Big(\int_{\mathbb{R}^3}\int_{\mathbb{R}^3}\f{|[F(s,y,\eta)-\mu]|^2}{\mu(\eta)} I_{\{|F(s,y,v)-\mu|\leq\mu\}}d\eta dv\Big)^{\f12}\nonumber\\
&\leq C\Big(\int_{\mathbb{R}^3}\int_{\mathbb{R}^3}\f{|[F(s,y,\eta)-\mu]|^2}{\mu(\eta)} I_{\{|F(s,y,v)-\mu|\leq\mu\}}d\eta dv\Big)^{\f12}\nonumber\\
&\leq \f{C}{(t-s)^{\f32}}\Big(\int_{\mathbb{R}^3}\int_{\mathbb{R}^3}\f{|[F(s,y,\eta)-\mu]|^2}{\mu(\eta)} I_{\{|F(s,y,v)-\mu|\leq\mu\}}d\eta dy\Big)^{\f12}\nonumber\\
&\leq \f{C}{(t-s)^{\f32}} \sqrt{\mathcal{E}(F_0)},
\end{align}
and
\begin{align}\label{4.12}
J_2&\leq C\Big(\int_{\mathbb{R}^3}\int_{\mathbb{R}^3}(1+|v|)^{-2\b+4}(1+|\eta|)^{4-\b}[1+\|w(\eta)F(s,y,\eta)\|_{L^\infty}]d\eta dv\Big)^{\f12}\nonumber\\
&\qquad\times\Big(\int_{\mathbb{R}^3}\int_{\mathbb{R}^3}|[F(s,y,\eta)-\mu]| I_{\{|F(s,y,v)-\mu|\geq\mu\}}d\eta dv\Big)^{\f12}\nonumber\\
&\leq C[1+\|w F(s)\|_{L^\infty}]^{\f12}\Big(\int_{\mathbb{R}^3}\int_{\mathbb{R}^3}|[F(s,y,\eta)-\mu]| I_{\{|F(s,y,v)-\mu|\geq\mu\}}d\eta dv\Big)^{\f12}\nonumber\\
&\leq \f{C}{(t-s)^{\f32}}[1+\|w F(s)\|_{L^\infty}]^{\f12}\Big(\int_{\mathbb{R}^3}\int_{\mathbb{R}^3}|[F(s,y,\eta)-\mu]| I_{\{|F(s,y,v)-\mu|\geq\mu\}}d\eta dy\Big)^{\f12}\nonumber\\
&\leq \f{C}{(t-s)^{\f32}}[1+\|w F(s)\|^{\f12}_{L^\infty}]\sqrt{\mathcal{E}(F_0)}\leq \f{C}{(t-s)^{\f32}}\sqrt{\mathcal{E}(F_0)},
\end{align}
where we have made a change of variable $v\rightarrow y=x-v(t-s)$ and used \eqref{4.2}, Proposition \ref{prop2.4} in \eqref{4.12-4} and \eqref{4.12}. Thus, it follows from \eqref{4.12-3}-\eqref{4.12}, for $\Omega=\mathbb{R}^3$, that 
\begin{align}\label{4.12-5}
\int_{\mathbb{R}^3}w(v)^{-1}|v|^2\left|\int_{\mathbb{R}^3}(1,\eta, |\eta|^2,\eta\otimes\eta)[F(s,y,\eta)-\mu]d\eta\right|dv\leq \f{C\sqrt{\mathcal{E}(F_0)}}{(t-s)^{\f32}}.
\end{align}
For the case $\Omega=\mathbb{T}^3$, we divide the integral into three parts:
\begin{align}\label{4.15}
&\int_{\mathbb{R}^3}w(v)^{-1}|v|^2\left|\int_{\mathbb{R}^3}(1,\eta, |\eta|^2,\eta\otimes\eta)[F(s,x-v(t-s),\eta)-\mu(\eta)]d\eta\right|dv\nonumber\\
&\leq \int_{|v|\geq N}\int_{\mathbb{R}^3}(1+|v|)^{-\b+2}(1+|\eta|^2)|[F(s,y,\eta)-\mu]|d\eta dv\nonumber\\
&\quad+\int_{|v|\leq N}\int_{\mathbb{R}^3}(1+|v|)^{-\b+2}(1+|\eta|^2)|[F(s,y,\eta)-\mu]|\cdot I_{\{|F(s,y,v)-\mu|\leq\mu\}}d\eta dv\nonumber\\
&\quad+\int_{|v|\leq N}\int_{\mathbb{R}^3}(1+|v|)^{-\b+2}(1+|\eta|^2)|[F(s,y,\eta)-\mu]| \cdot I_{\{|F(s,y,v)-\mu|\geq\mu\}}d\eta dv\nonumber\\
&=:J_3+J_4+J_5,
\end{align}
where $N>0$ is a large number to be chosen later. 
For terms on the right of \eqref{4.15}, one obtains that 
 \begin{align}\label{4.15-2}
J_3&\leq \int_{|v|\geq N}\int_{\mathbb{R}^3}(1+|v|)^{-\b+2}(1+|\eta|)^{-\beta+2}(1+\|w(\eta)F(s,y,\eta)\|_{L^\infty})d\eta dv\nonumber\\
&\leq C\sup_{0\leq s\leq t}(1+\|w(\eta)F(s,y,\eta)\|_{L^\infty})\int_{|v|\geq N}(1+|v|)^{-\b+2}dv\nonumber\\
&\leq C\int_{|v|\geq N}(1+|v|)^{-\b+2}dv\leq C N^{-\b+5},
\end{align}
\begin{align}\label{4.12-1}
J_4&\leq
 C\Big(\int_{|v|\leq N}\int_{\mathbb{R}^3}\f{|[F(s,y,\eta)-\mu]|^2}{\mu(\eta)} I_{\{|F(s,y,v)-\mu|\leq\mu\}}d\eta dv\Big)^{\f12}\nonumber\\
&\leq CN^{\f32}\Big(\int_{\mathbb{T}^3}\int_{\mathbb{R}^3}\f{|[F(s,y,\eta)-\mu]|^2}{\mu(\eta)} I_{\{|F(s,y,v)-\mu|\leq\mu\}}d\eta dy\Big)^{\f12}\nonumber\\
&\leq CN^{\f32}\sqrt{\mathcal{E}(F_0)},
\end{align}
and
\begin{align}\label{4.16}
J_5&\leq C[1+\|w F(s)\|_{L^\infty}]^{\f12}\Big(\int_{|v|\leq N}\int_{\mathbb{R}^3}|[F(s,y,\eta)-\mu]| I_{\{|F(s,y,v)-\mu|\geq\mu\}}d\eta dv\Big)^{\f12}\nonumber\\
&\leq CN^{\f32}\Big(\int_{\mathbb{T}^3}\int_{\mathbb{R}^3}|[F(s,y,\eta)-\mu]| I_{\{|F(s,y,v)-\mu|\geq\mu\}}d\eta dy\Big)^{\f12}\nonumber\\
&\leq CN^{\f32}\sqrt{\mathcal{E}(F_0)},
\end{align}
where once again we have made a change of variable $v\rightarrow y=x-v(t-s)$ and used \eqref{4.2}, Proposition \ref{prop2.4} in \eqref{4.12-1} and \eqref{4.16}.  Choosing $N^{-\b+5}=N^{\f32}\sqrt{\mathcal{E}(F_0)}$ and 
combining  \eqref{4.15} together with \eqref{4.15-2}-\eqref{4.16}, one obtains that for $\Omega=\mathbb{T}^3$, 
\begin{align}\label{4.17}
\int_{\mathbb{R}^3}w(v)^{-1}|v|^2\left|\int_{\mathbb{R}^3}(1,\eta, |\eta|^2,\eta\otimes\eta)[F(s,y,\eta)-\mu(\eta)]d\eta\right|dv\leq C\mathcal{E}(F_0)^{\f{\b-5}{2\b-7}}.
\end{align}
Now, substituting \eqref{4.11}, \eqref{4.12-5} and \eqref{4.17} into \eqref{4.10}, one has
\begin{align}\label{4.13}
&\Big|\int_{\mathbb{R}^3}(1,v,|v|^2, v\otimes v)[F(t,x,v)-\mu(v)]dv\Big|\nonumber\\
&\leq e^{-\f{t}{1-\nu}}\Big|\int_{\mathbb{R}^3}(1,v,|v|^2, v\otimes v)[F_0(x-vt,v)-\mu(v)]dv \Big|\nonumber\\
&\qquad+C\l\sup_{y\in\mathbb{R}^3,t-\l\leq s\leq t }\left|\int_{\mathbb{R}^3}(1,\eta, |\eta|^2,\eta\otimes\eta)[F(s,y,\eta)-\mu]d\eta\right|\\
&\qquad+C\l^{-\f32}\sqrt{\mathcal{E}(F_0)}+C\mathcal{E}(F_0)^{\f{\b-5}{2\b-7}}.\nonumber
\end{align}
Choosing $\l>0$ suitably small such that $C\l\leq \f12$, then the desired estimate \eqref{4.5} follows from \eqref{4.13}. Therefore, we complete the proof of Lemma \ref{lem4.2}.
\end{proof}

\begin{corollary}\label{cor4.3}
There is a constant $\v_0>0$ such that if 
\begin{equation*}
\mathcal{E}(F_0)+\sup_{t\geq t_1,x\in \Omega}e^{-\f{t}{1-\nu}}\Big|\int_{\mathbb{R}^3}(1,v,|v|^2, v\otimes v)[F_0(x-vt,v)-\mu(v)]dv \Big|\leq \v_0,
\end{equation*}	
then it holds that 
\begin{align*}
C_1^{-1}\leq \r(t,x),~T(t,x)\leq C_1,~~\mbox{and}~~|u(t,x)|\leq C_1,
\end{align*}
for all $t\geq t_1$ and $x\in\Omega$.
\end{corollary}


\noindent{\bf Proof of Theorem \ref{thm1.1}.} It follows immediately from Lemma \ref{lem3.1}, Lemma \ref{lem3.2}, Lemma \ref{lem4.1}, Lemma \ref{lem4.2} and Corollary \ref{cor4.3}; see also \cite{DHWY}. \qed

\medskip

In the end of the paper, we give a remark on the stability of solutions. For simplicity, we consider the only torus case $\Omega=\T^3$ and $(M_0,J_0,E_0)=(0,0,0)$, namely initial data have the same fluid quantities as the normalized global Maxwellian $\mu$. As in \cite{CJM}, due to the Csisz\'ar-Kullback inequality, it is direct to verify that for all $t>0$,
$$
\|F(t)-\mu\|_{L^1_{x,v}}\leq \sqrt{2 H(F(t)|\mu)}\leq  \sqrt{2 H(F_0|\mu)}=\sqrt{2\CE(F_0)}\leq \sqrt{2\v_0},
$$
where $H(\cdot|\cdot)$ is the relative entropy defined by
$$
H(F|G)=\int_{\R^3}\int_{\T^3} F\log \frac{F}{G}\,dxdv.
$$
Therefore, the solution is stable uniformly in time in the sense of $L^1_{x,v}$ under the assumptions of Theorem \ref{thm1.1}. Moreover, it would be also interesting to further study the large-time behavior of solutions obtained in Theorem \ref{thm1.1}. However, as the global-in-time existence is proved in the non-perturbation framework, it seems impossible to make use of the same idea as in \cite{DHWY} (also cf.~\cite{Yun1}) to justify that the difference $F(t,x,v)-\mu$ without adding the extra velocity weight $\mu^{-1/2}$ should approach zero in some sense as time goes to infinity. As one of possibilities, we will see if or not the method developed in \cite{GMM} could be adapted to  treat this problem in the future.

\medskip

\noindent{\it Acknowledgments.}  
Renjun Duan is partially supported by the General Research Fund (Project No. 409913) from RGC of Hong Kong. Yong Wang is partially supported by National Natural
Sciences Foundation of China No. 11401565.  Tong Yang is partially supported by the General Research Fund of Hong Kong, CityU 103412.

\end{document}